\documentclass[12pt, reqno,english,empty]{amsart}

\usepackage{array}

\usepackage{latexsym}
\usepackage{enumerate}
\usepackage{mathrsfs}
\usepackage{stmaryrd}
\usepackage{amsopn}
\usepackage{amsmath,amsthm,amssymb}
\usepackage{mathtools}
\usepackage{amsfonts}
\usepackage{soul}
\usepackage{amsbsy}
\usepackage{amscd,indentfirst,epsfig}
\usepackage{amsfonts,latexsym,verbatim,amsbsy}

\usepackage[colorlinks=true,linkcolor=blue,citecolor=blue,urlcolor=blue,breaklinks]{hyperref}

\usepackage[utf8]{inputenc}
\usepackage{mathrsfs}

\def\R{\mathbb R}

\def \t {\tau}

\def\d{\,\mathrm{d}}
\def \ddt{\frac{\mathrm{d}}{\mathrm{d}t}}

\def\z{\zeta}

\newtheorem{theo}{Theorem}[section]
\newtheorem{prop}[theo]{Proposition}

\newtheorem{nb}[theo]{Remark}
\def \leq {\leqslant}
\def \geq {\geqslant}

\numberwithin{equation}{section}

\def\beq{\begin{equation}}
\def\eeq{\end{equation}}
\def\beqn{\begin{equation*}}
\def\eeqn{\end{equation*}}
\def\bea{\begin{eqnarray}}
\def\eea{\end{eqnarray}}
\def\bean{\begin{eqnarray*}}
\def\eean{\end{eqnarray*}}
\def\bary{\begin{array}}
\def\eary{\end{array}}

\usepackage{breakurl}
\usepackage{natbib}
\usepackage{url}

\newcommand{\const}{\mathsf{const}}
\newcommand{\add}{\mathsf{add}}
\newcommand{\mult}{\mathsf{mult}}

\DeclarePairedDelimiter\norm{\Vert}{\rVert}
\DeclarePairedDelimiter\lnorm{\llbracket}{\rrbracket}

\title[Contractivity for Smoluchowski]
{\textbf{Contractivity for Smoluchowski's coagulation equation with
    solvable kernels}}

\author{José A. Cañizo}
\address{José A. Cañizo. Departamento de Matemática Aplicada,
  Universidad de Granada. Av. Fuentenueva S/N, 18071 Granada, Spain.}
\email{canizo@ugr.es}

\author{Bertrand Lods}
\address{Bertrand Lods. Università degli Studi
  di Torino \& Collegio Carlo Alberto, Department of Economics and
  Statistics, Corso Unione Sovietica, 218/bis, 10134 Torino, Italy.}
\email{bertrand.lods@unito.it}

\author{Sebastian Throm}

\address{Sebastian Throm. Departamento de Matemática Aplicada,
  Universidad de Granada. Av. Fuentenueva S/N, 18071 Granada, Spain.}
\email{throm@correo.ugr.es}

\begin{document}

\begin{abstract}
  We show that the Smoluchowski coagulation equation with the solvable
  kernels $K(x,y)$ equal to $2$, $x+y$ or $xy$ is contractive in
  suitable Laplace norms. In particular, this proves exponential
  convergence to a self-similar profile in these norms. These results
  are parallel to similar properties of Maxwell models for
  Boltzmann-type equations, and extend already existing results on
  exponential convergence to self-similarity for Smoluchowski's
  coagulation equation.
\end{abstract}

\maketitle

\section{Introduction}

Smoluchowski's coagulation equation describes the growth of clusters
in systems of merging particles in a broad range of
applications (see \cite{BLL} for general references on the matter). Precisely, the equation is given by
\begin{equation}\label{eq:Smol}\begin{split}
  \partial_{t}n(t,x)&=
  \frac{1}{2}\int_{0}^{x}K(x-y,y)n(t,x-y)n(t,y)\d y
  -\int_{0}^{\infty}K(x,y)n(t,x)n(t,y)\d y\\
  &=:\mathcal{C}(n(t,\cdot),n(t,\cdot))(x), \qquad x >0\end{split}
\end{equation}
where $n(t,x)$ is the density of clusters of size $x > 0$
at time $t \geq 0$ and the integral kernel $K(x,y) \geq 0$ describes
the rate at which clusters of sizes $x$ and $y$ merge. In
applications, the latter function is usually homogeneous of a certain
degree $\gamma$, i.e.~$K(ax,ay)=a^{\gamma}K(x,y)$ for all $a,
x,y>0$. In this paper we are concerned with the so-called
\emph{solvable kernels}: 
\begin{align*}
K(x,y)&=2 \qquad &\text{(constant kernel),}\\
K(x,y)&=x+y \qquad &\text{(additive kernel),}\\
K(x,y)&=xy \qquad  &\text{(multiplicative kernel).}
\end{align*}

For these
kernels an explicit solution to \eqref{eq:Smol} may be found by using
the Laplace transform, see \cite{MeP04,BLL}.

In this note we show that for these kernels, Equation \eqref{eq:Smol} satisfies new contractivity
properties in suitable weak distances which we define below. When one
considers the usual change of scale to self-similar variables, we show
that
\begin{equation*}
  \| g_1(\tau, \cdot) - g_2(\tau, \cdot) \|
  \leq e^{-\lambda t} \| g_1(0, \cdot) - g_2(0, \cdot) \|,
\end{equation*}
where $g_1 = g_1(\tau,x)$, $g_2 = g_2(\tau,x)$ are obtained from two
finite-mass solutions to \eqref{eq:Smol} through the change of
variables, and $\|\cdot\|$ is a suitable weighted norm of the Laplace
transform of $g$. Precise statements are given at the end of this
introduction. In particular, this provides explicit exponential rates
of convergence towards self-similarity with respect to this norm,
always for solutions with finite mass. This exponential convergence
was already known in other norms since \cite{CMM10,Sri11}, and in
fact our arguments have a similar flavour to those in
\cite{Sri11}. 

What is remarkable is that the calculation involving
these Laplace-based distances is much simpler, especially for the additive
kernel, and yields contractivity of the whole flow, not just of the
distance to self-similarity. These distances are inspired by analogous
norms based on the Fourier transform which have been used in
\cite{CarrTo} to study several models related to the Boltzmann
equation with constant collision kernels (the \emph{Maxwell cases}),
and which to our knowledge have not been exploited in
proving the convergence to self-similarity for coagulation equations.

Let us give some background on Smoluchowski's equation in order
describe our results more precisely. An important property
of~\eqref{eq:Smol} is the (formal) conservation of the total mass
$$M_{1}[n(t)]=M_{1}[n(0)] \qquad \forall t \geq 0$$
where, for any nonnegative function $f\::\R_{+}\to \R$, we set
$$M_{\ell}[f]=\int_{0}^{\infty}x^{\ell}f(x)\d x, \qquad \ell \geq 0.$$ 
Indeed, multiplying~\eqref{eq:Smol} by $x$ and integrating over
$(0,\infty)$ we obtain that $\ddt M_{1}[n(t)]=0$ by formally
interchanging the order of integration. However, for homogeneity
degree $\gamma>1$ this procedure cannot be made rigorous and solutions
in fact lose mass after some finite time, a phenomenon which is known
as \emph{gelation}, see for example \citet{EMP02,BF14,MR858257}, and
\citet[Chapter 9]{BLL} for a thorough discussion of this topic. In
fact, one defines the \emph{gelation time} $T_{*}$ as
\begin{equation*}
 T_{*}=\inf\{t\geq 0\;| M_{1}(t)<M_{1}(0)\}.
\end{equation*}
If $\gamma\leq 1$ one sets $T_{*}=\infty$.

A well-known conjecture, known as the \emph{scaling hypothesis},
states that the behaviour of solutions $n$ to~\eqref{eq:Smol} is
self-similar as $t\to T_{*}$ (perhaps under additional conditions on
the initial datum). That is: there exists a \emph{self-similar
  profile} $\widehat{n}$, a scaling function $s(t)\to \infty$ as
$t\to T_{*}$ and a constant $\alpha>0$ such that
\begin{equation}
  \label{eq:scaling:hyp}
  \bigl(s(t)\bigr)^{\alpha}n(t,s(t)x)
  \longrightarrow \widehat{n}(x)\qquad \text{if }t\to T_{*}.
\end{equation}
However, this claim is still unproven for most kernels $K$. The only
cases where~\eqref{eq:scaling:hyp} is well understood are the solvable
kernels $K(x,y)=2$, $K(x,y)=x+y$ and $K(x,y)=xy$. In fact, for these
rate kernels, the scaling hypothesis was verified in \cite{MeP04},
i.e.\@ there exists one unique fast decaying self-similar profile (up
to normalisation) which attracts all solutions with initial condition
satisfying $\int_{0}^{\infty}x^{\gamma+1}n(0,x)\d x<\infty$ with
respect to weak convergence. A more precise statement can be found in
\cite{MeP04}, where in addition the existence of \emph{fat-tailed
  profiles} was established and the corresponding domains of
attraction were characterised. These proofs heavily rely on Laplace
transform methods which allow to compute solution formulas for
\eqref{eq:Smol} rather explicitly.

The results on the fast-decaying profiles were further improved in
\cite{MeP06} by showing that~\eqref{eq:scaling:hyp} also holds
uniformly with respect to $x\in \R_{+}$ (i.e. ~\eqref{eq:scaling:hyp}
is obtained in $L^{\infty}(\R^{+})$). For the constant kernel, the
scaling hypothesis was also verified by a different approach which
relies on spectral gap estimates for the linearised coagulation
operator (\cite{CMM10}), yielding explicit rates of convergence to
self-similarity. Moreover, \cite{Sri11} later provided rates of
convergence for the primitive of $n$, for all three solvable kernels,
using explicit calculations inspired in arguments related to the
central limit theorem in probability.

The only known full verification of the scaling hypothesis for a class
of non-solvable kernels was recently given in \cite{CaT19} where
bounded perturbations $K(x,y)=2+\epsilon W(x,y)$ with
$\norm{W}_{L^{\infty}}\leq 1$ and small $\epsilon$ have been
considered. The proof again relies on spectral gap estimates and
provides explicit rates of convergence towards the self-similar
profile.

\medskip Assume the coagulation kernel has homogeneity degree
$\gamma$. It is known \citep{MeP04,EsM06} that if a finite-mass
solution satisfies \eqref{eq:scaling:hyp} with a finite-mass profile
$\widehat{n}$, it must happen that $\alpha = 2$. Introducing a change of unknown 
\begin{equation*}
  g(\tau, z) := s(t)^2 n \Big( t, s(t) z \Big), \qquad t=t(\tau) >0
\end{equation*}
and using that, for a coagulation kernel homogeneous of degree $\gamma$, 
\begin{multline*}
s(t)^{2} \mathcal{C}(n(t),n(t))(x)\bigg\vert_{t=t(\tau),x=s(t(\tau))z}=s(t(\tau))^{\gamma-1}\left[\mathcal{C}(g(\tau),g(\tau)\right](z)\\
\text{ and } 
\quad  z\partial_{z}g(\tau,z)=\bigg(s^{2}(t)\,x\partial_{x}n(t,x)\bigg)\bigg\vert_{t=t(\tau),x=s(t(\tau))z}\end{multline*}
we obtain that, if $n(t,x)$ satisfies \eqref{eq:Smol},
\begin{multline*}
\partial_{\tau}g(\tau,z)\\*
=\frac{\d t(\tau)}{\d \tau}\bigg\{2\dot{s}(t)\,s(t)\,n(t,x) + s^{2}(t)\mathcal{C}(n(t),n(t)) + s(t)\dot{s}(t)x \partial_{x}n(t,x)\bigg\}\bigg\vert_{t=t(\tau),x=s(t(\tau))z}\\
=\frac{\d t(\tau)}{\d\tau}\bigg\{2\frac{\dot{s}(t(\tau))}{s(t(\tau))}g(\tau,z)+s^{\gamma-1}(t(\tau))\mathcal{C}(g(\tau),g(\tau))(z)+\frac{\dot{s}(t(\tau))}{s(t(\tau))}z\partial_{z}g(\tau,z)\bigg\}.
\end{multline*}
where $\dot{s}$ denotes the derivative with respect to the original
variable $t$.  Therefore, choosing $s(t)$ and $t(\tau)$ such that
\begin{equation}
\label{eq:tts}
\frac{\d t(\tau)}{\d\tau}\frac{\dot{s}(t(\tau))}{s(t(\tau))}=1, \qquad
\qquad \frac{\d
  t(\tau)}{\d\tau}s^{\gamma-1}(t(\tau))=\frac{1}{k},
\end{equation}
(for any $k > 0$ to be chosen later) we obtain that $g(\tau,z)$
satisfies the self-similar Smoluchowski equation
\begin{equation}\label{eq:gta}
  \partial_\tau g(\tau,z) =
  2 g(\tau,z) + z \partial_z g(\tau,z) + \frac{1}{k} \mathcal{C}(g(\tau), g(\tau))(z), \qquad z >0, \qquad \tau >0
\end{equation}
and the stationary solutions to this equation are the self-similar
profiles $\widehat{n}=\widehat{n}(z)$ appearing in
\eqref{eq:scaling:hyp}. We refer to \cite{MeP06} and \cite{BLL} for
details on this subject. After solving, \eqref{eq:tts} yields
$$t(\tau)=\frac{\tau}{k}\qquad \text{ and }  \quad s(t)=e^{kt} \qquad \text{ if $\gamma=1$}$$
whereas
$$t(\tau)=\frac{1}{k (1-\gamma)} \left(
  e^{(1-\gamma) \tau} - 1 \right), \qquad \qquad s(t)=(1 + k(1-\gamma)
t)^{\frac{1}{1-\gamma}} \qquad \text{if $\gamma \neq 1$}.$$ Notice
that $s(t(\tau))=e^{\tau}$. The gelation time
$T_*$ is then equal to $+\infty$ if $\gamma \leq 1$, and is finite for
$\gamma > 1$. The self-similar profile $\widehat{n} = \widehat{n}(z)$
must then satisfy the equation
\begin{equation}
  \label{eq:ss-profile-eq}
  2 \widehat{n} + z \partial_z \widehat{n} + \frac{1}{k} \mathcal{C}(\widehat{n}, \widehat{n}) = 0.
\end{equation}
 In the three cases
which concern us in this paper, and always considering finite-mass
solutions, this becomes the following:
\begin{enumerate}
\item For the constant case $K = 2$ (so homogeneity $\gamma = 0$), the
  value of $k$ is irrelevant (since the convergence
  \eqref{eq:scaling:hyp} holds for all $k > 0$). We choose then
  $k = 1$ and obtain
  \begin{equation}\label{eq:const:selfsim}
    t(\tau) = e^\tau - 1,
    \qquad g(\tau, z) := e^{2\tau} n(e^{\tau}-1, e^\tau z),
  \end{equation}
  which satisfies then
  \begin{equation}\label{eq:Smol:const:selfsim}
    \partial_\tau g =
    2 g + z \partial_z g + \mathcal{C}_{\const}(g, g)
  \end{equation}
where $\mathcal{C}_{\const}$ is the coagulation operator in \eqref{eq:Smol} for $K=2$. For this model, mass is conserved for solutions to \eqref{eq:Smol:add:selfsim}, i.e.\@ $M_{1}[g(\t)]=M_{1}[g(0)]$ for all $\t\geq 0$. Moreover, we have 
\begin{equation*}
 \dfrac{\d}{\d\tau}M_{0}[g(\tau)]=M_{0}[g(\tau)]\,\left(1-M_{0}[g(\tau)]\right)
\end{equation*}
and thus, rescaling such that $M_{0}[g(0)]=1$, also the moment of order zero is conserved, i.e.\@ in summary we get
$$\int_{0}^{\infty}g(\tau,z) \left[\begin{array}{c}1 \\z \end{array}\right]\d z=\int_{0}^{\infty}g(0,z) \left[\begin{array}{c}1 \\z\end{array}\right]\d z \qquad \forall \tau >0.$$

\item For the linear case $K(x,y) = x + y$ (for which $\gamma=1$) assuming the solution $n$
  has mass $1$ requires that $k = 2$ in order to have a
  solution. Hence
  \begin{equation}\label{eq:add:selfsim}
    t(\tau) = \tfrac{1}{2} \tau,
    \qquad g(\tau, z) := e^{2\tau} n\Big(\tfrac{1}{2}\tau, e^{\tau} z \Big),
  \end{equation}
  which satisfies then
  \begin{equation}\label{eq:Smol:add:selfsim}
    \partial_\tau g =
    2 g + z \partial_z g + \tfrac{1}{2} \mathcal{C}_{\add}(g, g) \end{equation}
where $\mathcal{C}_{\add}$ is the coagulation operator in \eqref{eq:Smol} associated to $K(x,y)=x+y$. In that case, one sees that the first moment is conserved
$M_{1}[g(\t)]=M_{1}[g(0)]$ for any $\t >0$ whereas 
$$\dfrac{\d}{\d\tau}M_{2}[g(\tau)]=M_{2}[g(\tau)]\,\left(M_{1}[g(\tau)]-1\right).$$
This means that, if $M_{1}[g(0)]=1$ then both the first 
 and second moments are conserved for solutions to \eqref{eq:Smol:add:selfsim}, i.e.
\begin{equation}\label{eq:add:moments}
M_{1}[g(0)]=1 \Longrightarrow M_{1}[g(\t)]=1, \quad M_{2}[g(\t)]=M_{2}[g(0)], \qquad \forall \tau >0.\end{equation}

\item For the multiplicative case $K(x,y) = x y$ (corresponding to $\gamma=2$), assuming the
  solution $n$ has mass $1$ and initial second moment equal to $1$,
  then we must choose $k = 1$ in order to have the correct gelation time. Hence
  \begin{equation}\label{eq:mult:selfsim}
    t(\tau) = 1 - e^{-\tau},
    \qquad g(\tau, z) := e^{2\tau} n(1 - e^{-\tau}, e^\tau z),
  \end{equation}
  which satisfies then
  \begin{equation}\label{eq:Smol:mult:selfsim}
    \partial_\tau g =
    2 g + z \partial_z g + \mathcal{C}_{\mult}(g, g)
  \end{equation}
where $\mathcal{C}_{\mult}$ is the coagulation operator in \eqref{eq:Smol} associated to $K(x,y)=xy$.
  \end{enumerate}

Throughout this work, we will use the sub- and superscripts $\const$, $\add$
and $\mult$ to denote quantities related to the constant, additive and
multiplicative kernel respectively.

We introduce the following spaces
\begin{equation*}
 \begin{aligned}
  \mathbb{Y}_{\const}&=\left\{g \in L^{1}(\R^{+})\;;\int_{0}^{\infty}g(x)\d x=\int_{0}^{\infty}xg(x)\d x=0, \quad \int_{0}^{\infty}x^{2}|g(x)|\d x <\infty\right\}\\
  \mathbb{Y}_{\add}&=\left\{g \in L^{1}_{\text{loc}}(\R^{+})\;;\int_{0}^{\infty}xg(x)\d x=\int_{0}^{\infty}x^2g(x)\d x=0, \quad \int_{0}^{\infty}x^{3}|g(x)|\d x <\infty\right\}\\
  \mathbb{Y}_{\mult}&=\left\{g \in L^{1}_{\text{loc}}(\R^{+})\;;\int_{0}^{\infty}x^2g(x)\d x=\int_{0}^{\infty}x^3g(x)\d x=0, \quad \int_{0}^{\infty}x^{4}|g(x)|\d x <\infty\right\}.
 \end{aligned}
\end{equation*}
Finally, given $\kappa \in \R$, we define the following
\begin{equation}\label{eq:def:norm}
 \begin{gathered}
 \lnorm{G}_{\kappa}=\sup_{\eta >0}|\eta|^{-\kappa}\,|G(\eta)|\\
 \norm{g}_{\const,\kappa}=\lnorm{\mathcal{L}[g]}_{\kappa},\qquad \norm{g}_{\add,\kappa}=\lnorm{\mathcal{B}[g]}_{\kappa}\qquad \text{and}\qquad \norm{g}_{\mult,\kappa}=\lnorm{\mathcal{B}[xg]}_{\kappa}.
 \end{gathered}
\end{equation}
where $\mathcal{L}$ and $\mathcal{B}$ denote the Laplace and desingularised Laplace (Bernstein) transform, i.e.\@
\begin{equation*}
\mathcal{L}[g](\eta)=\int_{0}^{\infty}e^{-\eta\,x}g(x)\d x\quad \text{and}\quad \mathcal{B}[g](\eta)=\int_{0}^{\infty}(1-e^{-\eta\,x})g(x)\d x  \qquad \eta >0. 
\end{equation*}

We state here a first obvious result where the uniqueness property $\norm{u}=0$ comes from the fact that both $\mathcal{L}$ and $\mathcal{B}$ are one-to-one:
\begin{prop} The following holds
\begin{enumerate}
\item For $\kappa \in [0,2]$, $\norm{\,\cdot\,}_{\const,\kappa}$ is a norm on $\mathbb{Y}_{\const}$.
\item For $\kappa \in [0,3]$, $\norm{\,\cdot\,}_{\add,\kappa}$ is a norm on $\mathbb{Y}_{\add}$ and $\norm{\,\cdot\,}_{\mult,\kappa}$ is a norm on $\mathbb{Y}_{\mult}$.
\end{enumerate}
\end{prop}
\begin{nb} The constraint on the range of $\kappa$ for which the above
  quantities are norms is needed to ensure their finiteness (see
  \citet[Proposition 2.6]{CarrTo} for similar considerations on
  Fourier-like metrics).  Notice that, in the above listed cases, the
  spaces $(\mathbb{Y}_{\const},\norm{\,\cdot\,}_{\const,\kappa})$,
  $(\mathbb{Y}_{\add},\norm{\,\cdot\,}_{\add,\kappa})$ and
  $(\mathbb{Y}_{\mult},\norm{\,\cdot\,}_{\const,\kappa})$ are not
  necessarily Banach spaces. However, we will not need such a property
  in our analysis.
\end{nb}

Similar norms have also been used in \cite{NTV16,Thr17a} to prove the
uniqueness of self-similar profiles for perturbations of the constant
kernel.

Our main results assert the \emph{exponential contractivity} of the
above norm for the difference of two solutions to \eqref{eq:gta} in
the spirit of \citet[Sections 5.2 \& 5.3]{CarrTo}. Our method of proof
differs here from that of \cite{CarrTo}, who investigate contractivity
properties of the operator itself. We rather exploit a simple Duhamel
representation for the difference of two solutions $g_{1},g_{2}$ to
\eqref{eq:gta}.

One important point of the method is that it strongly exploits the
fact that, for all the three solvable kernels considered here,
\emph{two} different moments are conserved by the flow of
solution. This allows to \emph{fix} two such initial moments and work
in the various spaces $\mathbb{Y}_{\mult},$ $\mathbb{Y}_{\const}$ and
$\mathbb{Y}_{\add}$. More precisely, we prove here the following three
statements:

\begin{theo}[\textit{\textbf{constant kernel}}]\label{Thm:constant:kernel}
 Let $n_1(t,x)$ and $n_{2}(t,x)$ be solutions to~\eqref{eq:Smol} with constant kernel $K=2$ such that $n_{\ell}(0,\cdot)\in L^{1}(\R^{+})$ and 
 $$\int_{0}^{\infty}n_{\ell}(0,x)\d x=\int_{0}^{\infty}xn_{\ell}(x)\d x=1, \qquad \int_{0}^{\infty}x^2 n_{\ell}(x)\d x<\infty \qquad \text{ for } \: \ell=1,2.$$ For $\ell=1,2$ let furthermore  $g_{\ell}$ be the rescaling of $n_{\ell}$ as given by~\eqref{eq:const:selfsim}. Then, for each $\kappa\in(1,2]$ we have 
 \begin{equation*}
  \norm*{g_1(\tau,\cdot)-g_2(\tau,\cdot)}_{\const,\kappa}\leq \exp\left(-(\kappa-1)\tau\right)\norm{g_1(0,\cdot)-g_2(0,\cdot)}_{\const,\kappa}, \qquad \forall \tau \geq0.
 \end{equation*}
 In particular, this shows exponential convergence towards the unique self-similar profile $G_{\const}(x)=e^{-x}$ with respect to $\norm{\,\cdot\,}_{\const,\kappa}$.
\end{theo}

\begin{theo}[\textit{\textbf{additive kernel}}]\label{Thm:additive:kernel}
  Let $n_1(t,x)$ and $n_{2}(t,x)$ be solutions to~\eqref{eq:Smol} with additive kernel $K(x,y)=x+y$ such that $n_{\ell}(0,\cdot)\in L^{1}(\R^{+})$ and 
 $$\int_{0}^{\infty}xn_{\ell}(0,x)\d x=\int_{0}^{\infty}x^2 n_{\ell}(0,x)\d x=1, \qquad \int_{0}^{\infty}x^3 n_{\ell}(0,x)\d x<\infty \qquad \text{ for } \ell=1,2.$$ Let $g_{\ell}$ be the corresponding rescaling as specified in~\eqref{eq:add:selfsim}. Then, for each $\kappa\in(2,3)$ we have 
 \begin{equation*}
   \norm*{g_1(\tau,\cdot)-g_2(\tau,\cdot)}_{\add,\kappa}\leq \exp\left(-\tfrac{1}{2}(\kappa-2)\tau\right)\norm{g_1(0,\cdot)-g_2(0,\cdot)}_{\add,\kappa}, \qquad \forall \tau \geq0.
 \end{equation*}
 This shows in particular exponential convergence towards the unique self-similar profile 
 $$G_{\add}(x)=\frac{1}{\sqrt{2\pi}}x^{-3/2}e^{-x/2}$$ with respect to $\norm{\,\cdot\,}_{\add,\kappa}$.  
\end{theo}

For the multiplicative kernel, one can resort to a well-known change
of variables linking solutions to \eqref{eq:mult:selfsim} to the
solutions to \eqref{eq:add:selfsim} to deduce from Theorem
\ref{Thm:additive:kernel} the following
\begin{theo}[\textit{\textbf{multiplicative kernel}}]\label{Thm:multiplicative:kernel}
Let $n_{1}(t,x)$ and $n_{2}(t,x)$ be solutions to~\eqref{eq:Smol} with multiplicative kernel $K(x,y)=xy$ such that $n_{\ell}(0,\cdot)\in L^{1}(\R^{+})$ and 
$$\int_{0}^{\infty}x^2n_{\ell}(0,x)\d x=\int_{0}^{\infty}x^3 n_{\ell}(0,x)\d x=1 \qquad \int_{0}^{\infty}x^4 n_{\ell}(0,x)\d x<\infty \quad \text{ for } \ell=1,2.$$ Let $g_{\ell}$ be the rescaling of $n_{\ell}$ as given in~\eqref{eq:mult:selfsim}. Then, for each $\kappa\in(2,3)$ we have 
 \begin{equation*}
  \norm*{g_1(\tau,\cdot)-g_2(\tau,\cdot)}_{\mult,\kappa}\leq \exp\left(-\tfrac{1}{2}(\kappa-2)\tau\right)\norm{g_1(0,\cdot)-g_2(0,\cdot)}_{\mult,\kappa} \qquad \forall \tau\geq0.
 \end{equation*}
 In particular, this proves the convergence towards the unique self-similar profile 
 $$G_{\mult}(x)=\frac{1}{\sqrt{2\pi}}x^{-5/2}e^{-x/2}$$ with respect to $\norm{\,\cdot\,}_{\mult,\kappa}$.  
\end{theo}

The difference in the range of parameters for which each of the above
results holds is due to two different restrictions. The upper bound on
the allowed $\kappa$ is due to the choice of the conserved moments,
and ensures the finiteness of the respective Laplace-based norm
(recall for instance that,
$\norm{g_{1}(\tau)-g_{2}(\tau)}_{\const,\kappa} < \infty$ for
$\kappa \in (0,2]$ whereas
$\norm{g_{1}(\tau)-g_{2}(\tau)}_{\add,\kappa} < \infty$ for
$\kappa \in (0,3]$). More interestingly, the lower bound on the range
of $\kappa$ comes from the different behaviour of the semigroup
associated with the shifted drift operator
$g \mapsto z\partial_{z}g + 2g$ in the various spaces
$\mathbb{Y}_{\add},\mathbb{Y}_{\const},\mathbb{Y}_{\mult}.$

\subsection{Organization of the paper}

After this Introduction, Section \ref{sec:const} is devoted to the
proof of Theorem \ref{Thm:constant:kernel}, Section \ref{sec:add} is
devoted to the proof of Theorem \ref{Thm:additive:kernel} and Section
\ref{sec:mult} to that of Theorem \ref{Thm:multiplicative:kernel}.

\section{Proof for the constant kernel}\label{sec:const}

\begin{proof}[Proof of Theorem~\ref{Thm:constant:kernel}]
  Let $n$ be a solution to~\eqref{eq:Smol} with constant kernel $K=2$
  and let $g$ be the rescaled solution according
  to~\eqref{eq:const:selfsim}. It is easy to check that the
  corresponding Laplace transform
 \begin{equation*}
  N(t,\lambda)=\int_{0}^{\infty}\exp(-\lambda\,t)n(t,x)\d x, \qquad \lambda \in \R
 \end{equation*}
 satisfies the equation
 \begin{equation*}
   \partial_{t}N(t,\lambda)= N^{2}(t,\lambda)-2N(t,0)
   N(t,\lambda),
   \qquad\lambda \geq 0.
 \end{equation*}
Taking the limit $\lambda\to 0$ this yields the relation
\begin{equation*}
 \partial_{t}N(t,0)=-N^{2}(t,0), \qquad t \geq 0
\end{equation*}
where $N(t,0)=\int_{0}^{\infty}n(t,x)\d x$ is the moment of order zero. By assumption we have 
\begin{equation*}
 \int_{0}^{\infty}n(0,x)\d x=\int_{0}^{\infty}xn(0,x)\d x=1
\end{equation*}
which yields
\begin{equation}\label{eq:const:kernel:moment:zero}
 N(t,0)=\frac{1}{t+1} \qquad \forall t \geq 0.
\end{equation}
One directly verifies that 
\begin{equation*}
 \int_{0}^{\infty}g(\tau,z)\d z=e^{\tau}\int_{0}^{\infty}n(e^\tau-1,x)\d x\quad \text{and}\quad \int_{0}^{\infty}zg(\tau,z)\d z=\int_{0}^{\infty}xn(e^\tau-1,x)\d x\equiv 1.
\end{equation*}
Together with~\eqref{eq:const:selfsim} and~\eqref{eq:const:kernel:moment:zero} this gives
 \begin{equation}\label{eq:normalisation}
\int_{0}^{\infty}g(\tau,z)\d z=\int_{0}^{\infty}zg(\tau,z)\d z=1, \qquad \forall \tau \geq0 
\end{equation}
Denoting 
\begin{equation*}
 U(\tau,\eta)=\int_{0}^{\infty}g(\tau,z)\exp(-\eta z)\d z.
\end{equation*}
we have the relation
\begin{equation*}
 U(\tau,\eta)=e^{\tau}N(e^{\tau}-1,\eta e^{-\tau})
\end{equation*}
and $U$ satisfies the equation
\begin{equation}\label{eq:Up}
\partial_{\tau}U(\tau,\eta)+\eta\,\partial_{\eta}U(\tau,\eta)+U(\tau,\eta)=U^{2}(\tau,\eta)
\end{equation}
with initial datum $U_{0}(\eta)=U(0,\eta)$ while we also exploit that
$U(\tau,0)=\int_{0}^{\infty}g(\tau,z)\d z=1$ according
to~\eqref{eq:normalisation}.  For two solutions $n_1$ and $n_2$ with
rescalings $g_1$ and $g_2$ and corresponding Laplace transforms $U_1$
and $U_2$ we introduce $u(\tau,\eta)=U_1(\tau,\eta)-U_2(\tau,\eta)$
which solves the equation
\begin{equation}\label{up}
  \partial_{\tau}u(\tau,\eta)+\eta\,\partial_{\eta}u(\tau,\eta)+u(\tau,\eta)
  =u(\tau,\eta)\left( U_1(\tau,\eta)+U_2(\tau,\eta) \right)\end{equation}
We define the semigroup 
\begin{equation*}
 \mathbf{T}_{\tau}v(\eta)=e^{-\tau}v(\eta e^{-\tau})
\end{equation*}
so that 
\begin{equation}\label{eq:Duh1}
  u(\tau,\eta) =
  \mathbf{T}_{\tau}u_{0}(\eta)
  + \int_{0}^{\tau}\mathbf{T}_{\tau-s} \left[
    u(s,\cdot) \big( U_1(s,\cdot)+U_2(\tau, \cdot) \big)
  \right](\eta)\d s.
\end{equation}
One easily checks that 
\begin{equation*}
 \lnorm{\mathbf{T}_{\tau}v}_{\kappa}=\exp(-(1+\kappa)\tau)\lnorm{v}_{\kappa}
\end{equation*}
and thus, for the solution to \eqref{eq:Duh1}
\begin{equation*}
  \lnorm{u(\tau)}_{\kappa} \leq
  e^{-(1+\kappa)\tau}\lnorm{u_{0}}_{\kappa}
  + \int_{0}^{\tau}\, e^{-\left(1+\kappa\right)(\tau-s)}
  \lnorm{ u(s,\cdot) \big(U_1(s,\cdot)+U_2(\tau,\cdot)\big) }_{\kappa}\d s.
\end{equation*}
Since $U_{\ell}(\tau,\eta)\leq U(\tau,0)=1$ one has
\begin{equation*}
 \left|U_1(s,\eta)+U_2(s,\eta)\right| \leq 2 \qquad \forall \eta >0
\end{equation*}
so that
\begin{equation*}
  \lnorm{u(s,\cdot) \big( U(s,\cdot)+U_2(s,\cdot) \big)}_{\kappa}
  \leq 2\lnorm{u(s,\cdot)}_{\kappa}
\end{equation*}
and therefore
\begin{equation*}
  \lnorm{u(\tau)}_{\kappa} \leq
  e^{-(1+\kappa)\tau} \lnorm{u_{0}}_{\kappa}
  + 2\int_{0}^{\tau} \, e^{-\left(1+\kappa\right)(\tau-s)} \lnorm{u(s)}_{\kappa}\d s. 
\end{equation*}
From Gronwall's lemma applied to $w(t) := \lnorm{u(\tau)}_k
e^{(1+\kappa) \tau}$,
\begin{equation*}
\lnorm{u(\tau)}_{\kappa} \leq e^{(1-\kappa)t}\lnorm{u(0)}_{\kappa} 
\end{equation*}
so that exponential convergence holds for $\kappa > 1$.  Notice that
~\eqref{eq:normalisation} yields the relation
$u(t,0)=-\partial_{\eta}u(t,0)=0$ which ensures that
\begin{equation*}
\lnorm{u(t)}_{\kappa} < \infty \qquad \text{for $\kappa \in (0,2]$}
\end{equation*}
which gives contractivity for all $\kappa \in (1,2]$ and thus finishes the proof.
\end{proof}

\section{Proof for the additive kernel}\label{sec:add}

\begin{proof}[Proof of Theorem~\ref{Thm:additive:kernel}]

  Let $n$ be a solution to~\eqref{eq:Smol} with additive kernel
  $K(x,y)=x+y$ and $g$ the corresponding rescaled solution according
  to~\eqref{eq:add:selfsim} such that $\int_{0}^{\infty}xn(t,x)\d x=1$
  for all $t\geq 0$ (note that mass is conserved). Let us denote by
  $N$ and $U$ the corresponding desingularised Laplace (Bernstein)
  transforms, i.e.\@
\begin{equation*}
N(t,\lambda)=\int_{0}^{\infty}(1-e^{-x\lambda})n(t,x)\d x\quad \text{and}\quad U(\tau,\eta) = \int_{0}^{\infty}(1 - e^{-\eta z}) g(\tau,z)\d z,
\qquad \eta \in \R.
\end{equation*}
One easily checks that $N$ satisfies the equation
\begin{equation*}
 \partial_{t}N(t,\lambda)=-N(t,\lambda)+N(t,\lambda)(\partial_{\lambda}N)(t,\lambda),
\end{equation*}
since the total mass is normalised to one. In Laplace variables, the rescaling~\eqref{eq:add:selfsim} translates into $U(\t,\eta)=e^{\t}N(\tfrac{1}{2}\t,e^{-\t}\eta)$ such that $U$ solves
\begin{equation*}
\partial_{t}U(t,\eta)
=\tfrac{1}{2}\left[(U - 2 \eta) \partial_\eta U + U\right] \qquad \eta \geq 0.
\end{equation*}

Recall that, due to the choice of the initial condition, the equation
~\eqref{eq:add:selfsim} preserves first and second moments (see
\eqref{eq:add:moments}).  Let now $n_1$ and $n_2$ be two solutions
with corresponding rescalings $g_1$ and $g_2$ normalised according to
Theorem~\ref{Thm:additive:kernel} which yields
\begin{gather*}
  \int_0^\infty z g_1(\tau,z) \d z =   \int_0^\infty z g_2(\tau,z) \d z = 1
  \qquad \text{for all $\tau \geq 0$,}
  \\
  \int_0^\infty z^2 g_1(\tau,z) \d z =   \int_0^\infty z^2 g_2(\tau,x) \d z = 1
  \qquad \text{for all $t \geq 0$}.
\end{gather*}
Let $U_1$ and $U_2$ be the associated Bernstein transforms which consequently satisfy
\begin{gather*}
  U_1(\t,0) = U_2(\t,0) = 0
  \qquad \text{for all $\t \geq 0$,}
  \\
  \partial_\eta U_1(\t,0) = \partial_\eta U_2(\t,0) = 1
  \qquad \text{for all $\t \geq 0$,}
  \\
  \partial_\eta^2 U_1(\t,0) = \partial_\eta^2 U_2(\t,0) = -1
  \qquad \text{for all $t \geq 0$.}
\end{gather*}
Let $u(\t, \eta) = U_1(\tau,\eta)-U_2(\t, \eta)$ be the
difference of $U_1$ and $U_2$ which solves
\begin{equation}
  \label{up1}
  2\partial_{\t} u
  =
  (U_1 - 2 \eta) \partial_\eta u + u \partial_\eta U_2 + u.
\end{equation}
In order to rewrite \eqref{up1} we view it as an equation for $u$,
with coefficients which depend on $U_1$, $U_2$. We define the
characteristic curves $\tau \mapsto X(\t; \t_0, \eta_0)$ as the
solution to the ordinary differential equation
\begin{equation*}
  \frac{\mathrm{d}}{\mathrm{d}\tau}
  X := -\tfrac{1}{2}(U_1(\t, X) - 2 X),
  \qquad X(\t_0) = \eta_0.
\end{equation*}
The solution $u$ to \eqref{up1} can then be written as
\begin{equation*}
  u(\t,\eta) = u_0(X(0; \t, \eta)) \exp\left(\frac{1}{2}
    \int_0^{\t} (1 + \partial_\eta U_2(s, X(s; \t, \eta))) \d s.
  \right)
\end{equation*}
Using that
\begin{equation*}
  \partial_\eta U_2(\t, \eta)
  = \int_0^\infty z g_2(\t,z) e^{-\eta z} \d z
  \leq \int_0^\infty z g_2(\t,z) \d z  = 1
\end{equation*}
we have
\begin{equation*}
  u(\t,\eta) \leq u_0(X(0; \t, \eta)) e^\tau.
\end{equation*}
Hence,
\begin{multline*}
  \lnorm{ u(\t,\cdot)}_\kappa
  \leq e^{ \t} \sup_{\eta > 0} \frac{|u_0(X(0; \t,
    \eta))|}{|\eta|^\kappa}
  \\
  = e^{\t} \sup_{\eta > 0}
  \frac{|u_0(X(0; \t, \eta))|}{|X(0; \t, \eta)|^\kappa}
  \frac{|X(0; \t, \eta)|^\kappa}{|\eta|^\kappa}
  \leq \exp((1 - \frac{1}{2}\kappa) \t) \lnorm{u_0}_\kappa,
\end{multline*}
since
\begin{equation*}
  |X(0; \t, \eta)| \leq \eta e^{-\t/2}.
\end{equation*}
This last inequality can be seen as follows: using that
\begin{equation*}
  U_1(\t, \eta)
  =
  \int_{0}^{\infty}(1 - \exp(-\eta\,z)) g_1(\t,z) \d z
  \leq
  \int_{0}^{\infty} \eta z g_1(\t,z) \d z = \eta
\end{equation*}
we have
\begin{equation*}
  \frac{\mathrm{d}}{\mathrm{d}s}
  X(s; \t, \eta)
  =
  -\tfrac{1}{2}(U_1(\t, X(s; \t,\eta)) - 2 X(s; \t,\eta))
  \geq
  \tfrac{1}{2}X(s; \t,\eta),  
\end{equation*}
and hence
\begin{equation*}
  X(\t; \t, \eta) \geq X(0; \t,\eta) e^{\tfrac{1}{2}\t},
\end{equation*}
that is,
\begin{equation*}
  X(0; \t,\eta)\leq \eta e^{-\frac{1}{2}\t}.
\end{equation*}
We finally obtain
\begin{equation*}
  \lnorm{ u(\t,\cdot) }_\kappa
  \leq \exp((1 - \frac{1}{2}\kappa) \t) \lnorm{u_0}_\kappa,
\end{equation*}
which gives a contractivity for $2 < \kappa < 3$.
\end{proof}

\section{Proof for the multiplicative kernel}\label{sec:mult}

\begin{proof}[Proof of Theorem~\ref{Thm:multiplicative:kernel}]

  It is well known (see e.g.\@ \cite{MeP04}) that the choice
  $\int_{0}^{\infty}x^2 n(0,x)\d x=1$ fixes the gelation time to
  $T_{*}=1$ and additionally that solutions $n_{\add}$ and $n_{\mult}$
  to Smoluchowski's coagulation equation for the additive and
  multiplicative kernel respectively are related by the change of
  variables
  \begin{equation*}
    n_{\mult}(t,x)=\frac{1}{(1-t)x}n_{\add} \left(
      \log \Big(\frac{1}{1-t} \Big),x
    \right).
  \end{equation*}
  If we switch to self-similar variables the corresponding solutions
  $g_{\add}$ and $g_{\mult}$ respectively, satisfy the following
  relation:
\begin{equation*}
 g_{\mult}(\tau,z)=\frac{1}{z}g_{\add}(\tau,z).
\end{equation*}
Note that this change also transforms the time domain for $g_{\mult}$ to $(0,\infty)$.  We thus obtain that $z\,g_{\mult}(\tau,z)$ satisfies~\eqref{eq:Smol:add:selfsim} and consequently the second and third moment are preserved (since the second moment has been chosen to be one). 

We denote now by $U(\tau,\eta)=\int_{0}^{\infty}(1-e^{-\eta z})zg_{\mult}(z)\d z$ and for two solutions $g_{1,\mult}$ and $g_{2,\mult}$ we denote the difference $u(\tau,\eta)=U_{1}(\tau,\eta)-U_{2}(\tau,\eta)$. Thus, arguing exactly as for the additive kernel (where now the first moment is replaced by the second one) we again obtain
\begin{equation*}
  \lnorm{ u(\tau,\cdot) }_\kappa
  \leq \exp((1 - \tfrac{1}{2}\kappa) \tau) \lnorm{u_0}_\kappa,
\end{equation*}
i.e.\@  a contractivity for $2 < \kappa < 3$. We also note, that the rate of convergence to the self-similar profile at gelation time in the original time variable $t=1-e^{-\tau}$ is given by $(1-t)^{\kappa-2}$ as $t\to 1=T_{*}$. 
\end{proof}

\section*{Acknowledgements}
\label{ack}

We would like to thank José A.~Carrillo for suggesting the approach
using Fourier-based distances a long time ago. JAC and ST were
supported by project MTM2017-85067-P, funded by the Spanish government
and the European Regional Development Fund. ST has been funded by the
Deutsche Forschungsgemeinschaft (DFG, German Research Foundation) –
Projektnummer 396845724. BL gratefully acknowledges the financial
support from the Italian Ministry of Education, University and
Research (MIUR), ``Dipartimenti di Eccellenza'' grant 2018-2022. The
authors would like to acknowledge the support of the Hausdorff
Institute for Mathematics, since this work was started as a result of
their stay at the 2019 Trimester Program on kinetic theory.

\bibliographystyle{plainnat-linked}

\bigskip
\bigskip

\end{document}